\documentclass[oneside,a4paper]{article}
\usepackage{xcolor}
\usepackage[colorlinks=true,linkcolor=blue]{hyperref}
\newcommand{\markupdraft}[2]{
    \ifthenelse{\equal{#1}{display}}{#2}{}
    \ifthenelse{\equal{#1}{color}}{\color{#2}}{}
}

\newcommand{\newcolored}[3][]{{\markupdraft{color}{#2}#3}
    \ifthenelse{\equal{#1}{}}{}{\markupdraft{display}{{\color{yellow!70!black}[#1]}}}} 

\newcommand{\TODO}[2][]{\markupdraft{display}{{\color{red}~\\\noindent== TODO: #2 {\color{yellow}(#1)} ==\\}}}
\newcommand{\todo}[2][]{\markupdraft{display}{{\color{red}\noindent++TODO: #2 {\color{yellow}(#1)}++}}}
\renewcommand{\markupdraft}[2]{}

\usepackage{cmap}
\usepackage[utf8x]{inputenc}
\usepackage[T1]{fontenc}
\usepackage[english]{babel}
\usepackage{graphicx}
\usepackage{cite}
\usepackage{xcolor}
\usepackage{algorithm}
\usepackage{algpseudocode}
\usepackage{subcaption}

\usepackage{latexsym}
\usepackage{amsmath,amssymb} 
\usepackage{mathtools}       
\usepackage{stmaryrd}        
\usepackage{bm}              

\usepackage{amsthm}  
\newtheorem{theorem}{Theorem}[section]

\newtheorem{lemma}[theorem]{Lemma}
\theoremstyle{definition}

\providecommand{\argmin}{\operatornamewithlimits{argmin}} 
\DeclareMathOperator{\diag}{diag} 
\providecommand{\T}{\mathrm{T}} 
\renewcommand{\geq}{\geqslant} 
\renewcommand{\leq}{\leqslant} 
\DeclarePairedDelimiterX{\inner}[2]{\langle}{\rangle}{#1, #2}
\DeclarePairedDelimiter{\norm}{\lVert}{\rVert}

\providecommand{\bigO}{\mathcal{O}}

\begin{document}

\title{Fast Eigen Decomposition for Low-Rank Matrix Approximation}
\author{Youhei Akimoto \\ Faculty of Engineering, Shinshu University \\ Wakasato 4-17-1, Nagano, 380-8553, JAPAN \\ \url{y_akimoto@shinshu-u.ac.jp}}

\date{}
\maketitle

\begin{abstract}
  In this paper we present an efficient algorithm to compute the eigen decomposition of a matrix that is a weighted sum of the self outer products of vectors such as a covariance matrix of data. A well known algorithm to compute the eigen decomposition of such matrices is though the singular value decomposition, which is available only if all the weights are nonnegative. Our proposed algorithm accepts both positive and negative weights.
\end{abstract}
\tableofcontents

\newcommand{\eye}{\mathbf{I}}
\newcommand{\W}{\mathbf{W}}
\newcommand{\Lw}{\mathbf{L}_{w}}
\newcommand{\Sw}{\mathbf{S}_{w}}
\newcommand{\Rw}{\mathbf{R}_{w}}

\providecommand{\A}{\mathbf{A}}
\providecommand{\Qa}{\mathbf{Q}_{a}}
\providecommand{\B}{\mathbf{B}_a}
\providecommand{\Ea}{\mathbf{E}_{a}}
\providecommand{\Da}{\mathbf{D}_{a}}
\providecommand{\Eb}{\mathbf{E}_{b}}
\providecommand{\Db}{\mathbf{D}_{b}}
\providecommand{\X}{\mathbf{X}}
\providecommand{\Qc}{\mathbf{Q}_{c}}
\providecommand{\Bc}{\mathbf{B}_{c}}
\providecommand{\Xt}{\tilde{\mathbf{X}}}
\providecommand{\Qx}{\mathbf{U}_x}
\providecommand{\Rx}{\mathbf{R}_{x}}




\section{Introduction}

We focus on a positive definite symmetric matrix of the form
\begin{equation}
  \mathbf{A} = a \eye + \sum_{i=1}^{k} w_i \bm{x}_i \bm{x}_i^\T, \label{eq:a1}
\end{equation}
where a is a nonnegative real number, $w_i$ are real numbers, $\bm{x}_i$ are $m$ dimensional vectors. In a matrix form, we can rewrite it as
\begin{equation}
  \mathbf{A} = a \eye + \mathbf{X} \W \mathbf{X}^\T, \label{eq:a2}
\end{equation}
where $\mathbf{X} = (\bm{x}_1, \dots, \bm{x}_k)$ is a $m \times k$ matrix, $\W = \diag(w_1,\dots,w_k)$ is a $k \times k$ diagonal matrix. We assume $k \ll m$. Such a matrix appears in many situations, in particular, image and signal processing. In image processing, $\bm{x}_i$ may represent a vector of $m$ pixels of an image. Since the number of pixels, $m$, is in general large and it is often the case that we have a fewer number, $k$, of images than $m$.

Sometimes we need to compute the eigen decomposition of such a matrix. A well known application is the eigenface computation for face image recognition \cite{eigenface}, where the eigenvectors are the representative face images. 
It is well known that if $w_i$ are all nonnegative, instead of computing the decomposition of $\mathbf{A}$ directly with the cost of $\bigO(m^\omega)$, for some $\omega > 2$ and $\omega = 3$ in practice, we can compute the eigenvalue decomposition $\mathbf{A} = \Qa\Da\Qa^\T$ in $\bigO(mk^2)$ through the singular value decomposition (SVD) of a matrix $\mathbf{Y} = \X \sqrt{\W}$, where $\sqrt{\W}$ is a diagonal matrix whose diagonal elements are $\sqrt{w_i}$. Let $\mathbf{Y} = \mathbf{U}\mathbf{S}\mathbf{V}^\T$ be the thin SVD of $\mathbf{Y}$, i.e., $\mathbf{U}$ and $\mathbf{V}$ are $m \times k$ and $k \times k$ matrices with orthonormal columns, $\mathbf{S}$ is a diagonal matrix of dimension $k$. The thin SVD can be computed in $\bigO(mk^2)$. Given $\mathbf{Y} = \mathbf{U}\mathbf{S}\mathbf{V}^\T$, we have
\begin{equation}
  \mathbf{A} = a \eye + \mathbf{X} \W \mathbf{X}^\T = a \eye + \mathbf{Y} \mathbf{Y}^\T = a \eye + \mathbf{U} \mathbf{S} \mathbf{V}^\T \mathbf{V} \mathbf{S} \mathbf{U}^\T = a \eye + \mathbf{U} \mathbf{S}^2 \mathbf{U}^\T.
\end{equation}
From the above equality we have that the greatest $k$ eigenvalues of $\mathbf{A}$ are $a + s_i^2$ and the corresponding eigenvectors are $\bm{u}_i$. The other $m - k$ eigenvalues are all $1$ and the corresponding eigenvectors form a set of $m - k$ arbitrary unit vectors that are orthogonal to each other and orthogonal to all $\bm{u}_i$. However, if some of $w_i$ are negative, the above computation is not applicable anymore. We believe that there are many situations where one would like to incorporate both positive and negative weights in \eqref{eq:a1}, and that the efficient computation of the eigenvalues and the eigenvectors of such matrices will open a novel use of the principal component analysis for high dimensional data.

In this article we present a $\bigO(mk^2)$ computation of the eigen decomposition of a matrix of the form \eqref{eq:a2}, where $w_i$ can be both positive and negative. To the best of our knowledge, no such algorithm has been proposed in the literature, despite that the matrix of the form \eqref{eq:a2} often appears in applications. \todo{experiments}

\section{Motivating Example}

One of our motivating examples is as follows. We would like to maintain a positive definite symmetric matrix $\mathbf{A}_{t}$, with which we define a distance function $d_t(\bm{x}, \bm{0}) = \sqrt{\bm{x}^\T \mathbf{A}_t^{-1} \bm{x}}$ that is used to classify data into two. For each distance computation it costs $\bigO(m^2)$, which is problematic if $m$ is huge, in addition to the inversion of $\mathbf{A}_t$ at once. Therefore, we would like to restrict $\mathbf{A}_t$ to be of the form \eqref{eq:a2} that allows us to compute the distance in $\bigO(mk)$.

In the training phase, we receive a set of training data every iteration that are categorized as either ``regular'' or ``irregular''. To make the distances smaller for regular data and greater for irregular data, we update $\mathbf{A}_t$ as
\begin{equation}
  \tilde{\mathbf{A}}_{t+1} = \alpha \mathbf{A}_{t} + \beta \sum_{i=1}^{n} w_i \bm{x}_i \bm{x}_i^\T, \label{eq:a3}
\end{equation}
where $m$ is the number of data and $\bm{x}_i$ are the data received at one time, $w_i$ is positive if $\bm{x}_i$ is regular and negative if $\bm{x}_i$ is irregular, $\alpha$ and $\beta$ are some learning constants. The negative values for $w_i$ are important to learn irregularity of data actively. The updated matrix is still of the form \eqref{eq:a2}, but $k$ is incremented by $m$. To keep $k$ constant, we approximate $\tilde{\mathbf{A}}_{t+1}$ by a matrix $\mathbf{A}_{t+1}$ of the form \eqref{eq:a2} by solving
\begin{equation}
  \mathbf{A}_{t+1} = \argmin_{\mathbf{B}} \norm{ \ln(\tilde{\mathbf{A}}_{t+1}) - \ln(\mathbf{B}) }
\end{equation}
for $\mathbf{B}$ of the form \eqref{eq:a2} with the same $k$ as $\mathbf{A}_{t}$. Note that the matrix logarithm, $\ln$, is necessary for small eigenvalues to be enhanced and approximated. Otherwise, we will disregard the irregularity.
The solution to the above optimization problem is given by the eigen decomposition of $\tilde{\mathbf{A}}_{t+1}$ \cite{approx}. More precisely, given the eigenvalues $d_i$ of $\tilde{\mathbf{A}}_{t+1}$ that is sorted in the descending order, choose $\tau$ such that $\sum_{i=\tau+1}^{m-k+\tau} [\ln(d_i) - \sum_{i=\tau+1}^{m-k+\tau}\ln(d_i)]^2$. Then the solution to the above optimization problem is given by
\begin{equation}
  \mathbf{A}_{t+1}
  = \left(\prod_{i=\tau+1}^{m-k+\tau}d_i\right)^\frac{1}{m-k} \eye + \sum_{i=1}^{\tau} d_i \bm{q}_i\bm{q}_i^\T + \sum_{i=m-k+\tau+1}^{m} d_i \bm{q}_i\bm{q}_i^\T
  \enspace,
\end{equation}
where $\bm{q}_i$ is the eigenvector of $\tilde{\mathbf{A}}_t$ corresponding to $d_i$. Again, $\mathbf{A}_{t+1}$ is of the form \eqref{eq:a2}. However, every after the matrix update, one needs to compute the eigen decomposition of \eqref{eq:a3}, which costs $\mathcal{O}(m^3)$ in the naive computation. 

\section{Efficient Decomposition}
\providecommand{\QQ}{\mathbf{Q}}
\providecommand{\BB}{\mathbf{B}}
\providecommand{\BB}{\mathbf{B}}
\providecommand{\XX}{\mathbf{X}}
\providecommand{\YY}{\mathbf{Y}}
\providecommand{\nx}{n_x}
\providecommand{\ny}{n_y}

Our objective is to compute the eigenvalues and the eigenvectors of a matrix of the form \eqref{eq:a3} efficiently. Given the decomposition of $\mathbf{A}_{t} = \alpha \mathbf{I} + \mathbf{Q} \mathbf{B} \mathbf{Q}^\T$, where $\QQ$ is an $m \times n$ dimensional matrix with orthonormal columns and $\BB$ is an $n \times n$ dimensional symmetric matrix and $n < m$, the left hand side of \eqref{eq:a3} can be written as follows
\begin{equation}
  \mathbf{A} = \alpha \mathbf{I} + \mathbf{Q} \mathbf{B} \mathbf{Q}^\T + \X\X^\T - \mathbf{Y} \mathbf{Y}^\T \enspace,
  \label{eq:7}
\end{equation}
where $\X = (\sqrt{w_1} \bm{x}_1,\dots,\sqrt{w_{\nx}} \bm{x}_{\nx})$ is an $m \times \nx$ dimensional matrix and $\mathbf{Y} = (\sqrt{-w_{\nx+1}} \bm{x}_{\nx+1},\dots,\sqrt{-w_{\nx+\ny}} \bm{x}_{\nx+\ny})$ is an $m \times \ny$ dimensional matrix and $\nx$ and $\ny$ are the numbers of positive and negative weights, $w_i$, respectively.

The first step is transform $\QQ\BB\QQ^\T + \XX\XX^\T - \YY\YY$ into the form of $\QQ\BB\QQ^\T$, by applying the following lemma repeatedly.
\begin{lemma}\label{lem:blockdecom}
Let $\QQ$ is an $m \times n$ dimensional matrix with orthonormal columns and $\BB$ is an $n \times n$ dimensional symmetric matrix, $n < m$. Let $\X$ be an $m \times k$ dimensional matrix with rank $k \leq m$. Then, the following procedure compute in $\bigO(m(n+k)^2)$ the matrix decomposition of $\QQ\BB\QQ^\T \pm \X\X^\T$ in the form $\Qc\Bc\Qc^\T$, where $\Qc$ is an $m \times (n+k)$ dimensional matrix with orthonormal columns, $\Bc$ is an $n+k$ dimensional symmetric matrix. 
\begin{enumerate}
\item Compute $\Xt = \X - \QQ(\QQ^\T\X)$.
\item Compute the thin SVD of $\Xt = \Qx\mathbf{S}_x\mathbf{V}_x^\T$, where $\Qx$ is an $m \times k$ dimensional matrix with orthonormal columns, $\mathbf{S}_x$ is an $k \times k$ dimensional diagonal matrix, and $\mathbf{V}_x$ is an $k \times k$ orthogonal matrix.
\item Compute $\Rx = \mathbf{S}_x\mathbf{V}_x^\T$
\item Construct $\Qc = \begin{bsmallmatrix}\QQ & \Qx\end{bsmallmatrix}$, $\Bc = \begin{bsmallmatrix}\BB \pm (\QQ^\T\X)(\QQ^\T\X)^\T & \pm (\QQ^\T\X)\Rx^\T \\
  \pm \Rx(\QQ^\T\X)^\T & \pm \Rx\Rx^\T\end{bsmallmatrix}$
\end{enumerate}
\end{lemma}
\begin{proof}
The first step takes $\bigO(mnk)$, the second and third steps take $\bigO(mk^2)$ for the thin SVD. To construct $\Qc$, it requires $\bigO(m(m+k))$. To construct $\Bc$, it requires $\bigO(k(n+k)^2)$. Totally, the complexity is bounded by $\bigO(m(n+k)^2)$. 

Since a matrix $\eye - \QQ\QQ^\T$ maps a $m$ dimensional column vector into the subspace orthogonal to the column space of $\QQ$, the left singular vectors of $\Xt = (\eye - \QQ\QQ^\T)\X$, i.e., the columns of $\Qx$, must be orthogonal to each column of $\QQ$. This proves that the columns of $\Qc$ are orthonormal to each other. The symmetry of $\Bc$ is trivial from the construction.

The correctness is obvious from the following equality
\begin{align*}
  \MoveEqLeft[2]\QQ\BB\QQ^\T \pm \X\X^\T \\
  &= \QQ \BB \QQ^\T \pm (\QQ\QQ^\T\X + \Qx\Rx)(\QQ\QQ^\T\X + \Qx\Rx)^\T \\
  &= \begin{aligned}[t]
      \QQ (\BB  \pm \QQ^\T\X\X^\T\QQ)\QQ^\T &\pm \QQ\QQ^\T\X\Rx^\T\Qx^\T \\
      &\pm \Qx\Rx\X^\T\QQ\QQ^\T \pm \Qx\Rx\Rx^\T\Qx^\T 
      \end{aligned}\\
  &= 
  \begin{bmatrix}
  \QQ & \Qx
  \end{bmatrix}
  \begin{bmatrix}
  \BB \pm (\QQ^\T\X)(\QQ^\T\X)^\T & \pm (\QQ^\T\X)\Rx^\T \\
  \pm \Rx(\QQ^\T\X)^\T & \pm \Rx\Rx^\T
  \end{bmatrix}
  \begin{bmatrix}
  \QQ^\T \\ \Qx^\T
\end{bmatrix}
  \\
  &= \Qc \Bc \Qc^\T
  \enspace. \qedhere
\end{align*}
\end{proof}

Applying the above lemma twice, we can transform $\mathbf{Q} \mathbf{B} \mathbf{Q}^\T + \X\X^\T - \mathbf{Y} \mathbf{Y}^\T$ into the form $\Qa\B\Qa^\T$, where $\Qa$ is an $m \times (n + \nx + \ny)$ dimensional matrix with orthonormal columns and $\B$ is an $(n + \nx + \ny)$ dimensional symmetric matrix. One can easily obtain the eigen decomposition of $\Qa\B\Qa^\T$, as stated in the following lemma.
\begin{lemma}\label{lem:block2eig}
Let $\A$ be an $m \times m$ dimensional matrix with a decomposition $\A = \Qa \B \Qa^\T$, where $\Qa$ is an $m \times k$ dimensional matrix with orthonormal columns and $\B$ is an $k \times k$ dimensional symmetric matrix, $k < m$. Then, the thin eigen decomposition $\A = \Ea\Da\Ea^\T$ is computed in $\bigO(mk^2)$, where $\Da$ is a diagonal matrix of dimension $k$ and $\Ea$ is an $m \times k$ dimensional matrix with orthonormal columns, by the following procedure.
\begin{enumerate}
\item Compute the eigen decomposition $\B = \Eb\Db\Eb^\T$.
\item Construct $\Ea = \Qa \Eb$ and $\Da = \Db$.
\end{enumerate}
\end{lemma}

\begin{proof}
  First, the eigen decomposition of a symmetric matrix is unique up to a permutation of the diagonal elements of the eigenvalue matrix. Therefore, it is sufficient to show that $\Ea$ is a matrix with orthonormal columns, $\Da$ is a diagonal matrix, and $\A = \Ea \Da \Ea^\T$, where the last two are trivial from the construction. The orthogonality of $\Ea$ is confirmed by checking $\Ea^\T\Ea = \Eb^\T \Qa^\T \Qa \Eb = \eye$. The computational complexity is $\mathcal{O}(k^3)$ for the first step and $\mathcal{O}(mk^2)$ for the second step. This ends the proof.
\end{proof}

\providecommand{\Ev}{\mathbf{E}}
\providecommand{\Sv}{\mathbf{S}}
\providecommand{\Yp}{\mathbf{Y}_{+}}
\providecommand{\Ym}{\mathbf{Y}_{-}}
\providecommand{\Qone}{\mathbf{Q}_1}
\providecommand{\Bone}{\mathbf{B}_1}
\providecommand{\Qtwo}{\mathbf{Q}_2}
\providecommand{\Btwo}{\mathbf{B}_2}
\providecommand{\Qthr}{\mathbf{Q}_3}
\providecommand{\Dthr}{\mathbf{D}_3}
\providecommand{\mup}{k_{+}}
\providecommand{\mum}{k_{-}}

Now we are ready to prove the following main theorem.
\begin{theorem}\label{thm:act-dec}
  Let $\QQ$ be an $m \times n$ dimensional matrix with orthonormal columns, $\BB$ be an $n \times n$ dimensional symmetric matrix and $n < m$, $\XX$ be an $m \times \nx$ dimensional matrix, $\YY$ be an $m \times \ny$ dimensional matrix and $\alpha$ is a real number.
  
The eigenvalues of a matrix $\alpha \eye + \QQ\BB\QQ^\T + \XX\XX^\T - \YY\YY$ and the eigenvectors corresponding to the eigenvalues besides $\alpha$ are $(\alpha + [\Dthr]_{i,i})$ and $([\Qthr]_{:, i})$, for $i=1,\dots,n+\nx+\ny$, which are computed in $\bigO(m(n+\nx+\ny)^2)$ in the following way.
\begin{enumerate}
\item Apply Lemma~\ref{lem:blockdecom} to $\QQ \BB \QQ^\T + \XX\XX$ and obtain its decomposition $\Qone\Bone\Qone^\T$.
\item Apply Lemma~\ref{lem:blockdecom} to $\Qone\Bone\Qone^\T - \YY\YY$ and obtain its decomposition $\Qtwo\Btwo\Qtwo^\T$.
\item Apply Lemma~\ref{lem:block2eig} to $\Qtwo\Btwo\Qtwo^\T$ and obtain its eigen decomposition $\Qthr\Dthr\Qthr^\T$.
\end{enumerate}
\end{theorem}

\paragraph{Implementation}
The algorithm is implemented in Fortran with LAPACK \cite{lapack} and compiled to Python module using \textsc{f2py} accompanying \textsc{scipy} library. Our algorithm calls the LAPACK subroutines \textsc{dsyev} once and \textsc{dgesvd} twice. Except the input and output arguments, we reserve $3 (n + r + m + 1)(r + m)$ floating point number array as a working space, including the working space for \textsc{dsyev}, \textsc{dgesvd}, and internal matrix manipulations. If $3 (n + r + m + 1)(r + m) \geq n^2$, we recommend to use \textsc{dsyev} instead of our proposed algorithm. The source code is available on GitHub Gist \cite{effeig}.

\section{Experiments}

Figure~\ref{fig:scaling} shows (a) the CPU time scaling over the number of rows, $m$, with different $n = n_x = n_y$, and (b) the CPU time divided by $m(n + \nx + \ny)^2$. As references we also plot the CPU time of the eigenvalue decomposition through the SVD for $\nx = 3$ and $\nx = 3 \lfloor m/3\rfloor$ ($n = \ny = 0$) and the CPU for the naive eigen decomposition for $\nx = 3 \lfloor m/3\rfloor$ for the proposed algorithm and the naive algorithm, respectively.  The matrices $\QQ$, $\BB$, $\XX$, and $\YY$ are generated randomly. The experiment has been conduced with Python 3.6.0 (Numpy version 1.11.3) on macOS Sierra (Processor: 2.6 GHz Intel Core i7 processor, Memory: 16 GB 2133 MHz LPDDR3). The Fortran source code \cite{effeig} is compiled with \textsc{f2py} version 2.

The right figure indicates that the CPU time of \textsc{feigh} asymptotically scales in $\mathcal{O}(m(n + n_x + n_y)^2)$ as the theory tells. Comparing the CPU time for \textsc{feigh} with $n = \nx = \ny = 1$ and the CPU time for \textsc{dgesvd} with $\nx = 3$ and $n = \ny = 0$, \textsc{feigh} costs more time by the factor of about $1.3$ for $m > 10^7$. It tells that we loose only the factor of $1.3$ to treat the negative sign in \eqref{eq:a1}. On the other hand, the CPU time for \textsc{feigh} with $n = \nx = \ny = \lfloor m / 3 \rfloor$ is an almost same scaling as \textsc{dsyev} with $\nx = \lfloor m / 3 \rfloor$ ($n = \ny = 0$) and a slightly better scaling than \textsc{dgesvd} with $\nx = \lfloor m / 3 \rfloor$ ($n = \ny = 0$).

\TODO{
  \begin{itemize}
  \item comparison with svd ($\nx = 3$, $\nx = M$), eig ($\nx = M$)
  \item complexity of eig and svd. svd $mn^2$, eig: $m^3$ in practice, but $m^\omega$ for $2 < \omega < 2.3...$ in theory.
  \item To read: \url{http://www.sciencedirect.com/science/article/pii/S0047259X07000887}
  \item To read: \url{http://theory.stanford.edu/~tim/s15/l/l9.pdf}
  \end{itemize}
}

\begin{figure}[t]
  \centering
  \begin{subfigure}{0.5\hsize}
    \includegraphics[width=\hsize]{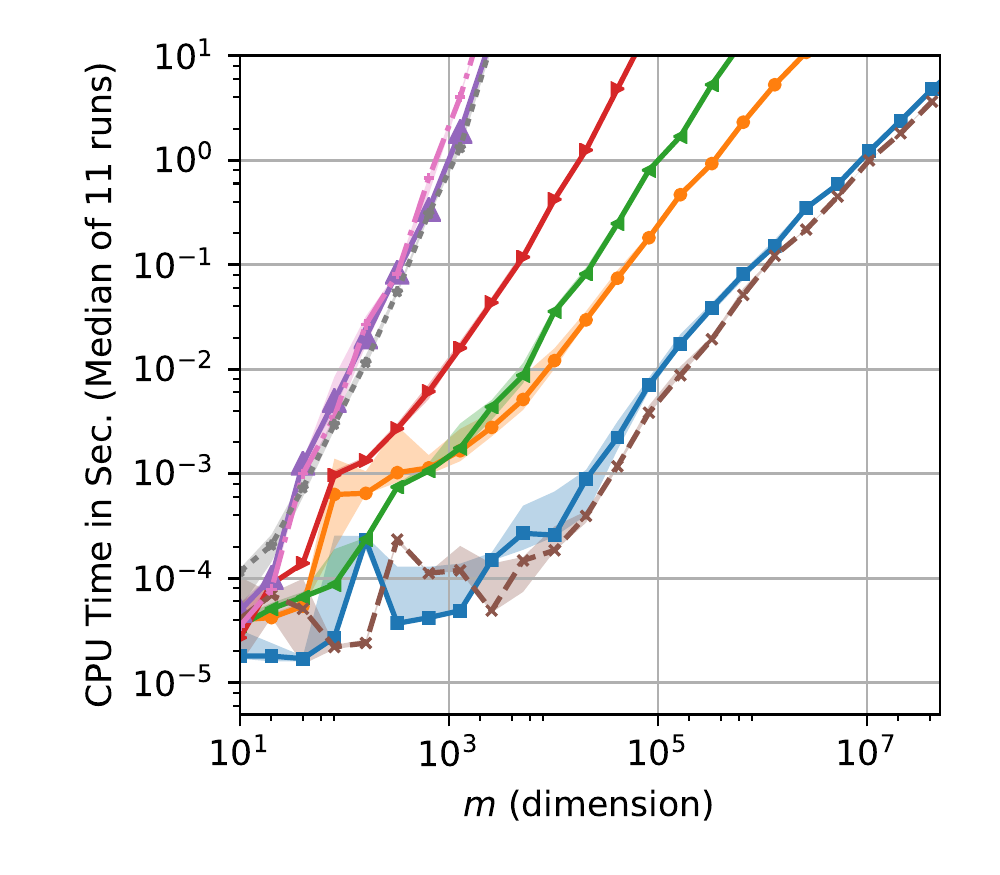}
  \end{subfigure}%
  \begin{subfigure}{0.5\hsize}
    \includegraphics[width=\hsize]{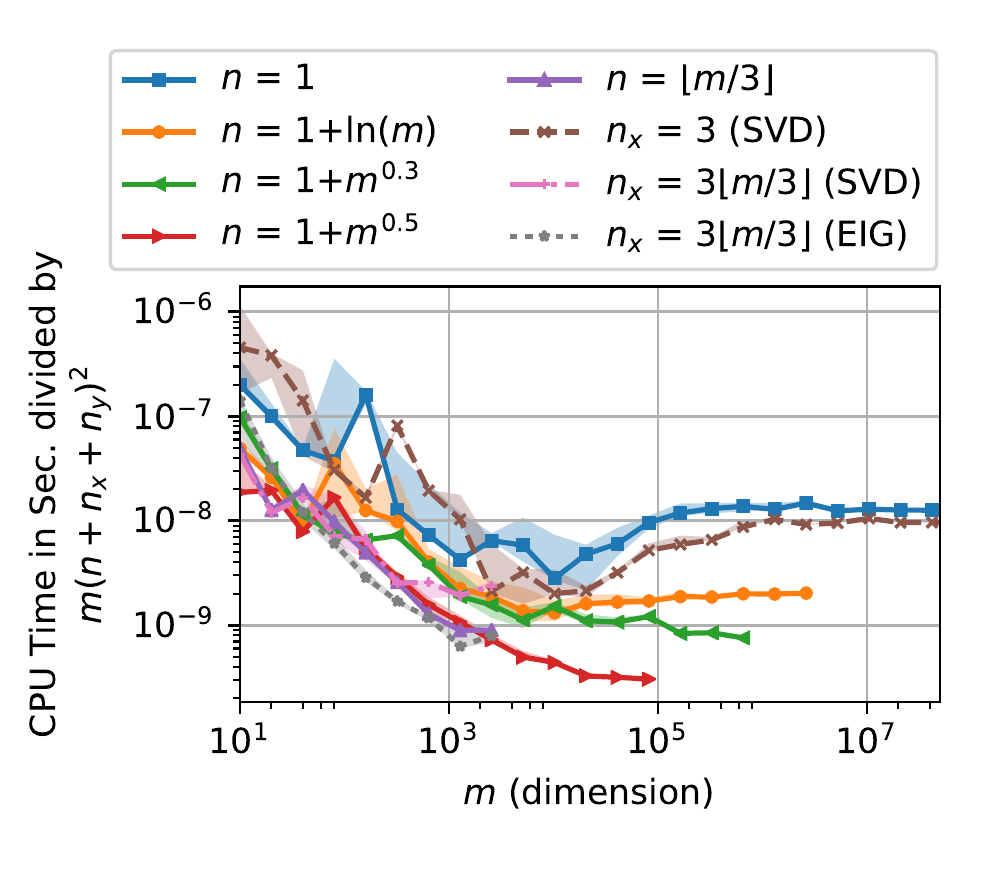}
  \end{subfigure}
  \caption{CPU time scaling. The median and the 10\%-90\% interval over 11 runs is reported for the proposed algorithm \textsc{feigh}, LAPACK routine \textsc{dgesvd} (labeled as SVD), and LAPACK routine \textsc{dsyev} (labeled as EIG) . }
  \label{fig:scaling}
\end{figure}


\end{document}